\documentclass[11pt,a4paper]{amsart}
\usepackage[all]{xy}
\usepackage{amssymb}

\numberwithin{equation}{section}

\newcommand{\wt}[1]{\widetilde{#1}}

\newtheorem{theorem}{Theorem}[section]
\newtheorem*{theorem*}{Theorem}

\newtheorem{lemma}[theorem]{Lemma}
\newtheorem{proposition}[theorem]{Proposition}

\newtheorem*{conjecture*}{Conjecture}

\newtheorem{remark}[theorem]{Remark}

\newcommand{\opname}[1]{\operatorname{\mathsf{#1}}}

\renewcommand{\mod}{\opname{mod}\nolimits}

\newcommand{\add}{\opname{add}\nolimits}

\newcommand{\der}{\cd}

\newcommand{\ind}{\opname{ind}}

\newcommand{\rad}{\opname{rad}\nolimits}

\newcommand{\Z}{\mathbb{Z}}

\newcommand{\Hom}{\opname{Hom}}

\newcommand{\Ext}{\opname{Ext}}

\newcommand{\End}{\opname{End}}

\newcommand{\cc}{{\mathcal C}}
\newcommand{\cd}{{\mathcal D}}

\newcommand{\ct}{{\mathcal T}}

\renewcommand{\hat}[1]{\widehat{#1}}

\begin{document}

\title[Hall polynomials]{Hall polynomials for representation-finite cluster-tilted algebras}
\date{\today}
\author{Changjian Fu}
\address{Department of Mathematics\\
SiChuan University\\
610064 Chengdu\\
P.R.China
}
\email{
\begin{minipage}[t]{5cm}
changjianfu@scu.edu.cn
\end{minipage}}
\thanks{Partially supported by NSF of China(No.11001185).}
\begin{abstract}
We show the existence of Hall polynomials for representation-finite cluster-tilted algebras.
\end{abstract}
\maketitle

\section{Introduction}

\subsection{} Let $k$ be a finite field and $\Lambda$ a locally bounded $k$-algebra, that is, $\Lambda$ is an associative algebra and $\Lambda$ has a set of primitive orthogonal idempotents $\{e_i\}_I$ such that $\Lambda=\oplus_{i,j\in I}e_i\Lambda e_j$, and both $\dim_k\Lambda e_i$ and $\dim_ke_i\Lambda$ are finite for all $i\in I$. Let $\mod \Lambda$ be the category of right $\Lambda$-modules with finite length. For $L, M,N\in \mod \Lambda$, we denote by $F^M_{N,L}$ the number of submodules $U$ of $M$ such that $U\cong L$ and $M/U\cong N$.

Let $E$ be a field extension of $k$. For any $k$-space $V$, we denote by $V^E$ the $E$-space $V\otimes_k E$. Clearly, $\Lambda^E$ naturally becomes a $E$-algebra. The field $E$ is called {\it conservative}~\cite{Ringel93}  for $\Lambda$ if for any indecomposable $M\in \mod \Lambda$, $(\End M/\rad \End M)^E$ is a field.  Set
\[\Omega=\{E| E~ \text{is a finite field extension of $k$ which is conservative for $\Lambda$}\}.
\]
For a given $\Lambda$ with $\Omega$ infinite,  the algebra $\Lambda$ {\it has Hall polynomials} provided that for any $L,M,N\in \mod \Lambda$, there exists a polynomial $\phi^M_{N,L}\in \Z[T]$ such that for any conservative finite field extension $E$ of $k$ for $\Lambda$,
\[\phi^M_{N,L}(|E|)=F^{M^E}_{N^E,L^E}.
\]
We call $\phi^M_{N,L}$ the {\it Hall polynomial} associated to $L,M,N\in \mod \Lambda$. Note that if $\Lambda$ is representation-finite, then $\Omega$ is an infinite set.

It has been conjectured by Ringel~\cite{Ringel93} that any representation-finite algebra has Hall polynomials. This conjecture has been verified for representation-directed algebras by Ringel~\cite{Ringel93}, cyclic serial algebras by Guo~\cite{Guo95} and Ringel~\cite{Ringel93b} and some other classes of algebras({\it cf. eg}. ~\cite{GuoPeng97}).
\subsection{}
Let $A$ be a finite-dimensional hereditary algebra over a field $k$. Let $\mod A$ be the finitely generated right $A$-modules and $\der^b(\mod A)$ the bounded derived category with suspension functor $\Sigma$. The cluster category $\cc(A)$ associated with $A$ was introduced in ~\cite{BuanMarshReinekeReitenTodorov}(independently in ~\cite{CalderoChaptonSchiffler} for $A_n$ case) as the orbit category $\der^b(\mod A)/\tau^{-1}\circ \Sigma$, where $\tau$ is the Auslander-Reiten translation of $\der^b(\mod A)$. A cluster-tilting object $T$ in $\cc(A)$ is an object such that $\Ext^1_{\cc(A)}(T,T)=0$ and it is maximal with this property. The endomorphism algebra $\End_{\cc(A)}(T)$ of a cluster-tilting object $T$ is called the cluster-tilted algebra of $T$, which were first introduced by Buan, Marsh and Reiten in ~\cite{BuanMarshReiten07}. Among others, they showed that cluster-tilted algebras are Gorenstein of dimension $1$. This has been further generalized to a more general setting by Keller-Reiten~\cite{KellerReiten07} and K\"{o}nig-Zhu~\cite{KonigZhu08}. Moreover, Keller and Reiten~\cite{KellerReiten07} proved  that the stable Cohen-Macaulay category of a given cluster-tilted algebra(more generally, $2$-Calabi-Yau tilted algebra) is $3$-Calabi-Yau. A direct consequence of the stably Calabi-Yau property is that the Auslander-Reiten conjecture holds true for cluster-tilted algebras. Namely, let $B$ be a cluster-tilted algebra over a field $k$, if $M$ is a finitely generated right $B$-module such that  $\Ext^i_{\mod B}(M, M\oplus B)=0$ for all $i\geq 1$, then $M$ is a projective $B$-module.

In ~\cite{Zhu2010}, Zhu has introduced certain Galois coverings for cluster categories and cluster-tilted algebras({\it cf.} also ~\cite{AssemBrustleSchiffler09}), called repetitive cluster categories and repetitive cluster-tilted algebras respectively( for the precisely definition, {\it cf.} Section~\ref{s:repetitive cluster categories}). We refer to ~\cite{BongartzGabriel82} for the notions of (Galois) covering functors.
The aim of this note is to show that Ringel's conjecture holds true for representation-finite repetitive cluster-tilted algebras. In particular, representation-finite cluster-tilted algebras have Hall polynomials.
\begin{theorem}~\label{t: main theorem}
Let $\Lambda$ be a representation-finite repetitive cluster-tilted algebra over a finite field $k$. Then $\Lambda$ has Hall polynomials.
\end{theorem}
Let us mention here that a variant proof of this theorem may be applied to  generalized cluster-tilted algebras of higher cluster categories of type ADE. In order to make this note concise, we restrict ourselves to the case of cluster-tilted algebras.
After recall some basic definitions and properties of repetitive cluster-tilted algebras in Section~\ref{s:repetitive cluster categories} , we will give the proof of Theorem~\ref{t: main theorem} in Section~\ref{s: proof main theorem}.

\section{Repetitive cluster categories and repetitive cluster-tilted algebras}~\label{s:repetitive cluster categories}
\subsection{}Let $\der$ be a $k$-linear triangulated category with suspension functor $\Sigma$ and $\ct$ a functorially finite subcategory of $\der$. The subcategory $\ct$ is called a {\it cluster-tilting subcategory}, if the followings are equivalent:
\begin{itemize}
\item[$\circ$] $X\in \ct$;
\item[$\circ$] $\Ext^1_{\der}(X, \ct)=0$;
\item[$\circ$] $\Ext^1_{\der}(\ct, X)=0$.
\end{itemize}
An object $T\in \der$ is a {\it cluster-tilting object} if and only if $\add T$ is a cluster-tilting subcategory.
A cluster-tilting object $T\in \der$ is called {\it basic} provided that  $T=\oplus_{i=1}^nT_i$, where $T_i, i=1, \cdots, n$ are indecomposable and $T_i\not\cong T_j$ whenever $i\neq j$.

Let $\ct$ be a cluster-tilting subcategory of $\der$ and $\mod \ct$ the category of finitely presented right $\ct$-modules.
It has been proved by K\"{o}nig and Zhu in~\cite{KonigZhu08}({\it cf.} also ~\cite{KellerReiten07}) that the functor $\Hom_{\der}(\ct,-): \der \to \mod \ct$ induces an equivalence $\der/\add \Sigma \ct\cong \mod \ct$. Moreover, if $\der$ has Auslander-Reiten triangles, then the Auslander-Reiten sequences of $\mod \ct$ is induced from the Auslander-Reiten triangles of $\der$. In this case, we have $\tau^{-1}\Sigma \ct=\ct$, where $\tau$ is the Auslander-Reiten translation of $\der$.
\subsection{}
Let $A$ be a finite-dimensional hereditary algebra over a field $k$. Let $\mod A$ be the category of finitely generated right $A$-modules and $\der^b(\mod A)$ the bounded derived category  with suspension functor $\Sigma$. Let $\tau$ be the Auslander-Reiten translation of $\der^b(\mod A)$. In the following, we fix a positive integer $m$ and  set $F:=\tau^{-1}\circ \Sigma$. The {\it repetitive cluster category} $\cc_{F^m}(A)$ introduced by Zhu~\cite{Zhu2010} is the  orbit category of $\der^b(\mod A)/<F^m>$, which is by definition  a $k$-linear category whose objects are the same as $\der^b(\mod A)$, and whose morphisms are given by
\[\cc_{F^m}(A)(X,Y):=\bigoplus_{i\in \Z}\Hom_{\der^b(\mod A)}(X, F^{mi}Y), \text{where}\ X, Y\in \der^b(\mod A).
\]
When $m=1$, we get the  cluster category $\cc(A)$. By the main theorem of Keller~\cite{Keller05}, we know that $\cc_{F^m}(A)$ admits a canonical triangle structure such that the canonical projection functor $\pi_m:\der^b(\mod A)\to \cc_{F^m}(A)$ is a triangle functor. Moreover, by the universal property of the orbit category $\cc_{F^m}(A)$, we have a triangle functor $\rho_m: \cc_{F^m}(A)\to \cc(A)$ such that $\pi_A=\rho_m\circ \pi_m$, where $\pi_A$ is the canonical projection functor $\pi_A: \der^b(\mod A)\to \cc(A)$.

It has been shown in ~\cite{Zhu2010} that there exists bijections between the following three sets: the set of cluster-tilting subcategory of $\der^b(\mod A)$, the set of cluster-tilting subcategories of $\cc_{F^m}(A)$ and the set of cluster-tilting subcategories of $\cc(A)$, via the triangle functors: $\pi_m:\der^b(\mod A)\to \cc_{F^m}(A)$ and $\rho_m:\cc_{F^m}(A)\to \cc(A)$. In particular, the repetitive cluster categories have cluster-tilting objects.

Let $\wt{T}$ be a basic cluster-tilting object in the repetitive cluster category $\cc_{F^m}(A)$, the endomorphism algebra $\End_{\cc_{F^m}}(A)(\wt{T})$ is called the {\it repetitive cluster-tilted algebra} of $T$.
We have the following main results of ~\cite{Zhu2010}({\it cf.} Theorem 3.7 and Theorem 3.8 in~\cite{Zhu2010}).
\begin{theorem}~\label{t:repetitive cluster}
Let $T$ be a basic cluster-tilting object in $\cc(A)$ and $A=\End_{\cc(A}(T)$ the cluster-tilted algebra of $T$. Let $\wt{T}$ be the corresponding basic cluster-tilting object in $\cc_{F^m}(A)$ of $T$ via the triangle functor $\rho_m$ and
 $\wt{A}=\End_{\cc_{F^m}(A)}(\wt{T})$  the associated  repetitive cluster-tilted algebra. Then we have the followings:
\begin{enumerate}
\item the restriction of $\pi_m:\der^b(\mod A)\to \cc_{F^m}(A)$ induces a Galois covering $\pi_m:\pi_A^{-1}(\add T)\to \add \wt{T}$ of $\wt{A}$. Moreover, the projection functor $\pi_m:\der^b(\mod A)\to \cc_{F^m}(A)$ induces a push-down functor $\wt{\pi}_m:\der^b(\mod A)/ \pi^{-1}_A(\add \Sigma T)\to \mod \wt{A}$;
    \item the functor $\rho_m:\cc_{F^m}(A)\to \cc(A)$ restricted to the cluster-tilting subcategory $\add \wt{T}=\rho^{-1}_m(\add T)$ induces a Galois covering of $A$. Moreover, the functor $\rho_m$ also induces a push-down functor $\wt{\rho}_m: \mod \wt{A}\to \mod A$.
\end{enumerate}
\end{theorem}

\begin{remark}~\label{r:locally bonded and directed}
Set $\ct=\pi_A^{-1}(\add T)$
and let $\ind \ct$ be a set of representatives of the isoclasses of all indecomposable objects in $\ct$.
Set \[\End(\ct):=\bigoplus_{T_i, T_j\in \ind \ct}\Hom_{\der^b(\mod A)}(T_i,T_j),\]
which is an associative algebra without units.
It is not hard to see that $\End(\ct)$ is locally bounded.
  On the other hand, the finitely presented $\ct$-modules coincides with $\End(\ct)$-modules of finite length.
  Hence, we have an equivalence of categories  \[\mod \End(\ct)\cong \der^b(\mod A)/\Sigma \ct\] by ~\cite{KonigZhu08}, which implies that $\End(\ct)$ is directed.
\end{remark}
\begin{remark}
One can verify that the functor $\wt{\pi}_m$ coincides with the restriction of the push-down functor induced by the Galois covering $\pi_m:\pi_A^{-1}(\add T)\to \add \wt{T}$. Hence, $\wt{\pi}_m$ is an exact functor. On the other hand, any exact sequence of $\der^b(\mod A)/\Sigma \ct$ can be lifted to a triangle in $\der^b(\mod A)$ ({\it cf.} Lemma $8$ of ~\cite{Palu08}). Then one can also prove the exactness  directly in this setting.
\end{remark}
Note that a Galois covering functor will not induce a Galois covering for the corresponding categories of modules in general. However, for the  Galois coverings in the above theorem,  we have the following observation.
\begin{proposition}~\label{p:covering}
Keep the notations in Theorem~\ref{t:repetitive cluster}, we have the followings:
\begin{itemize}
\item[(1)]the push-down functor $\wt{\pi}_m:\der^b(\mod A)/ \pi^{-1}_A(\add \Sigma T)\to \mod \wt{A}$ is a Galois covering of $\mod \wt{A}$;
    \item[(2)] the push-down functor $\wt{\rho}_m:\mod \wt{A}\to \mod A$ is a Galois covering of $\mod A$.
\end{itemize}
\end{proposition}
\begin{proof}
We will only prove the first statement and the second one follows  similarly.

 Let $\ct=\pi_A^{-1}(\add T)$. Since $\Sigma \ct$ is also a cluster-tilting subcategory of $\der^b(\mod A)$, we have $F\Sigma \ct=\Sigma\ct$. One shows that $F^m:\der^b(\mod A)\to \der^b(\mod A)$ induces a $k$-linear equivalence $F^m:\der^b(\mod A)/\Sigma \ct\to \der^b(\mod A)/\Sigma \ct $. Let $G$ be the infinite cyclic group generated by $F^m$ which is acting freely on the objects of $\der^b(\mod A)/\Sigma \ct$.
 Note that by Theorem~\ref{t:repetitive cluster}, we have the following commutative diagram
\[\xymatrix{\der^b(\mod A)\ar[d]_{Q_1}\ar[rr]^{\pi_m}& &\cc_{F^m}(A)\ar[d]_{Q_2}\\
\der^b(\mod A)/\Sigma \ct\ar[rr]^{\wt{\pi}_m}& &\cc_{F^m}(A)/\Sigma \wt{T}=\mod \wt{A}}
\]
where $Q_1$ and $Q_2$ are natural quotient functors.
Then for any indecomposable objects $X,Y$ in $\der^b(\mod A)/\Sigma \ct$, one can show that $\wt{\pi}_m(X)\cong \wt{\pi}_m(Y)$ if and only if there exists $i\in \Z$ such that $Y\cong F^{im}X$ in $\der^b(\mod A)/\Sigma\ct$. On the other hand, for any objects $M,N$ in $\mod \wt{A}$ with preimages $\hat{M}, \hat{N}$ in $\der^b(\mod A)/\Sigma \ct$ respectively. We clearly have
\[\Hom_{\mod \wt{A}}(M,N)=\bigoplus_{i\in \Z}\Hom_{\der^b(\mod A)/\Sigma \ct}(\hat{M}, (F^m)^i\hat{N})=\bigoplus_{i\in \Z}\Hom_{\der^b(\mod A)/\Sigma \ct}((F^m)^i\hat{M}, \hat{N}).
\]
This particular implies that $\mod \wt{A}$ identifies the orbit category of $\der^b(\mod A)/\Sigma \ct$ by $G$.
It is easy to see that the functor $\wt{\pi}_m$ coincides with the quotient functor.
\end{proof}

\section{Poorf of the main theorem}~\label{s: proof main theorem}

To prove our main result, we need the following result of Guo and Peng~\cite{GuoPeng97}, which gives a sufficient condition for the existence of Hall polynomials
\begin{lemma}~\label{l:guopeng-result}
Let $k$ be a finite field.
Let $\Lambda$ be a   finite-dimensional $k$-algebra of representation-finite type  and  there is a locally bounded $k$-algebra $R$ which is directed, such that there exists a covering funtor $F:\mod R\to \mod \Lambda$, and for any $M, N\in \mod \Lambda$, there exist $X,Y\in \mod R$ with $FX=M$ and $FY=N$ such that $F$ induces the $k$-isomorphism
\[\Ext^1_R(X,Y)\cong \Ext^1_{\Lambda}(M,N).
\]
Then $\Lambda$ has Hall polynomials.
\end{lemma}
\begin{remark}~\label{r:weaker condition}
According to Theorem $5.1$ of ~\cite{GuoPeng97}, the last condition in Lemma~\ref{l:guopeng-result} can be weakened to the following: for any $M,N\in \mod \Lambda$ with $N$ indecomposable, there exist $X_i,Y_i\in \mod R, i=1, 2$ with $FX_i\cong M$ and $FY_i\cong N$ such that $F$ induces the $k$-isomorphisms
\[\Ext_{R}^1(X_1, Y_1)\cong \Ext_{\Lambda}^1(M,N)~\text{and }~\Ext^1_{R}(Y_2, X_2)\cong \Ext^1_{\Lambda}(N,M).
\]
\end{remark}
 Now we are in a position to  prove the main theorem of this note.

\subsection*{Proof of Theorem~\ref{t: main theorem}}
By the definition of repetitive cluster-tilted algebras, there is a finite-dimensional hereditary algebra $A$ such that $\Lambda$ is the endomorphism algebra of a basic cluster-tilting object $\wt{T}$ in a repetitive cluster category $\cc_{F^m}(A)$ of $A$. Note that we have an equivalence of categories $\Hom_{\cc_{F^m}(A)}(\wt{T},-): \cc_{F^m}(A)/\Sigma \wt{T}\to \mod \Lambda$. Hence, $\Lambda$ is representation-finite implies that $A$ is representation-finite. Let $\ct=\pi_m^{-1}(\add \wt{T})$,
 by remark~\ref{r:locally bonded and directed}, we know that $\End(\ct)$ is locally bounded and directed.
By part $(1)$ of Proposition~\ref{p:covering}, we deduce that \[\wt{\pi}_m:\mod \End(\ct)\to \mod \Lambda\] is a directed Galois covering of $\mod \Lambda$.

According to Lemma~\ref{l:guopeng-result} and Remark~\ref{r:weaker condition}, it suffices to show that for any $M,N\in \mod \Lambda$ with $N$ indecomposable, there exist $X_i,Y_i\in \mod \End(\ct), i=1, 2$ with $FX_i\cong M$ and $FY_i\cong N$ such that $F$ induces the $k$-isomorphisms
\[\Ext_{\End(\ct)}^1(X_1, Y_1)\cong \Ext_{\Lambda}^1(M,N)~\text{and }~\Ext^1_{\End(\ct)}(Y_2, X_2)\cong \Ext^1_{\Lambda}(N,M).
\]
We will only prove the existence for the first isomorphism, where the second one follows similarly.
We may and we will assume that $M$ is also indecomposable  and $\Ext^1_{\Lambda}(M,N)\neq 0$.

Recall that  we have the following commutative diagram
\[\xymatrix{\der^b(\mod A)\ar[d]^{Q_1}\ar[rrr]^{\pi_m}&&& \cc_{F^m}(A)\ar[d]^{Q_2}\\
\mod \End (\ct)\cong\der^b(\mod A)/\Sigma \ct \ar[rrr]^{\wt{\pi}_m}& &&\cc_{F^m}(A)/\Sigma \wt{T}\cong\mod \Lambda}
\]
where $Q_1$ and $Q_2$ are natural quotient functors. Since $\wt{\pi}_m$ is a Galois covering, there exists $\hat{M}, \hat{N}\in \mod \End (\ct)$ such that $\wt{\pi}_m(\hat{M})=M$ and $\wt{\pi}_m(\hat{N})=N$. Moreover, $\{F^{im}\hat{M}|i\in \Z\}$ forms a complete set of preimages of $M$ in $\mod \End(\ct)$. For simplicity, we set $\hat{M}_i=F^{im}\hat{M}$ in the following.
By abuse of notations, we still denote by $\hat{M}_i, \hat{N}\in \der^b(\mod A)$ the unique  indecomposable preimages of $\hat{M}_i$ and $\hat{N}$ respectively.
Without of loss generality, we may assume that $\hat{N}\in \mod A$.
 By the definition of covering functor and the fact that $\wt{\pi}_m$ is an exact functor preserving projectivity, we have the following $k$-isomorphism
\[\bigoplus_{i\in \Z} \Ext^1_{\End(\ct)}(\hat{M}_i, \hat{N})\cong \Ext^1_{\Lambda}(M,N).
\]
Since $0\neq \dim_{k}\Ext^1_{\Lambda}(M,N)<\infty$, the vector space $\Ext^1_{\End(\ct)}(\hat{M}_i, \hat{N})$ vanishes for all but finitely many $i$.
Let $t\in \Z$ such that $\Ext^1_{\End(\ct)}(\hat{M}_t, \hat{N})\neq 0$ and $\Ext^1_{\End(\ct)}(\hat{M}_j, \hat{N})=0$ for $j>t$.

We claim that $\Ext^1_{\End(\ct)}(\hat{M}_t, \hat{N})\cong \Ext^1_{\Lambda}(M,N)$. It suffices to show that $\Ext^1_{\End(\ct)}(\hat{M}_i, \hat{N})=0$ for $i<t$. By the Auslander-Reiten translation formula, we have
\[\Ext^1_{\End(\ct)}(\hat{M}_i,\hat{N})\cong D\overline{\Hom}_{\End(\ct)}(\hat{N}, \tau \hat{M}_i),
\]
where $\tau$ is the Auslander-Reiten translation of $\mod \End(\ct)$ which is induced by the Auslander-Reiten translation of $\der^b(\mod A)$.
On the other hand, we have
\[\overline{\Hom}_{\End(\ct)}(\hat{N}, \tau \hat{M}_i)=\frac{\Hom_{\der^b(\mod A)}(\hat{N}, \tau \hat{M}_i)}{\{f:\hat{N}\to \tau \hat{M}_i~ \text{factoring through } \add \Sigma \ct\
 \text{or} \add \Sigma^2\ct\} }.
\]
Note that $\Ext^1_{\End(\ct)}(\hat{M}_t, \hat{N})\neq 0$ implies that $\Hom_{\der^b(\mod A)}(\hat{N}, \tau \hat{M}_t)\neq 0$. Since $\hat{N}\in \mod A$, we deduce that
$\tau\hat{M}_t\in \mod A$ or $\tau\hat{M}_t\in \Sigma \mod A$. Recall that $\hat{M}_i=F^{im}\hat{M}$,
if $\tau\hat{M}_t\in \mod A$, then we have
\[\Hom_{\der^b(\mod A)}(\hat{N}, \tau \hat{M}_i)=\Hom_{\der^b(\mod A)}(\hat{N}, F^{(i-t)m}\tau \hat{M_t})=0\ \text{for}\ i<t.
\]
Now assume that $\tau\hat{M}_t\in \Sigma \mod A$ and let $\tau\hat{M}_t=\Sigma L$, where $L\in \mod A$. From
 \[0\neq\Hom_{\der^b(\mod A)}(\hat{N}, \tau \hat{M}_t)=\Hom_{\der^b(\mod A)}(\hat{N}, \Sigma L),\]
 we  deduce  that $L$ is a predecessor of $\hat{N}$ in $\der^b(\mod A)$.
In this case, we have
\begin{eqnarray*}
\Hom_{\der^b(\mod A)}(\hat{N}, \tau \hat{M}_i)&=&\Hom_{\der^b(\mod A)}(\hat{N}, F^{(i-t)m}\tau \hat{M}_t)\\
&=&\Hom_{\der^b(\mod A)}(\hat{N}, \tau^{(t-i)m}\Sigma^{(i-t)m+1}L)\\
&=&\begin{cases}0& \text{for}\ i<t-1;\\
\Hom_{\der^b(\mod A)}(\hat{N}, \tau^m\Sigma^{1-m} L)& \text{for}\ i=t-1.
\end{cases}
\end{eqnarray*}
It is clear that $\Hom_{\der^b(\mod A)}(\hat{N}, \tau^m\Sigma^{1-m} L)=0$ if $m\geq 2$. Now suppose that $m=1$ and
 $\Hom_{\der^b(\mod A)}(\hat{N}, \tau L)\neq 0$, then $\hat{N}$ is a predecessor of $\tau L$ and hence a predecessor of $L$, which contradicts to the fact that $\der^b(\mod A)$ is directed. We have proved that $\Hom_{\der^b(\mod A)}(\hat{N}, \tau \hat{M}_i)=0$ for $i\neq t$ and hence $\Ext^1_{\End(\ct)}(\hat{M}_i, \hat{N})=0$ for $i\neq t$, which completes the proof.

\def\cprime{$'$}
\providecommand{\bysame}{\leavevmode\hbox
to3em{\hrulefill}\thinspace}
\providecommand{\MR}{\relax\ifhmode\unskip\space\fi MR }
\providecommand{\MRhref}[2]{%
  \href{http://www.ams.org/mathscinet-getitem?mr=#1}{#2}
} \providecommand{\href}[2]{#2}

\end{document}